\documentclass[12pt,a4paper]{amsart}
\usepackage[top=35mm, bottom=35mm, left=30mm, right=30mm]{geometry}
\usepackage[colorlinks=true,citecolor=blue]{hyperref}
\usepackage{mathptmx}
\usepackage{eucal}
\usepackage{graphicx}
\usepackage{mathrsfs}
\usepackage{amssymb}
\usepackage{amsmath}
\usepackage{amsthm}
\usepackage{amscd}
\usepackage{xcolor}
\usepackage{verbatim}

\usepackage{color}

\newtheorem{thm}{Theorem}[section]

\newtheorem{lem}[thm]{Lemma}

\theoremstyle{definition}

\numberwithin{equation}{section}

\newcommand{\N}{\mathbb{N}}

\newcommand{\Z}{\mathbb{Z}}

\newcommand{\ep}{\epsilon}

\newcommand{\ra}{\rightarrow}

\numberwithin{equation}{section}

\input xy
\xyoption{all}

\begin{document}

\title{discrete spectrum for amenable group actions}
\author{Tao Yu, Guohua Zhang and Ruifeng Zhang}

\address[T. Yu]{Department of Mathematics,
Shantou University, Shantou 515063, Guangdong, China -- and -- Shanghai Center for Mathematical Sciences, Fudan University, Shanghai 200433, China}
\email{ytnuo@mail.ustc.edu.cn}

\address[G. H. Zhang]{School of Mathematical Sciences and Shanghai Center for Mathematical Sciences, Fudan University, Shanghai 200433, China}
\email{chiaths.zhang@gmail.com}

\address[R. F. Zhang]{School of Mathematics, Hefei University of Technology, Hefei 230009, Anhui, China}
\email{rfzhang@hfut.edu.cn}

\begin{abstract}
 In this paper, we study
 discrete spectrum of invariant measures
for countable discrete amenable group actions.

We show that an invariant measure has discrete spectrum if and only if it has bounded measure complexity. We also prove that, discrete spectrum can be characterized via measure-theoretic complexity using names of a partition and the
Hamming distance, and it turns out to be equivalent to both mean equicontinuity
and equicontinuity in the mean.
\end{abstract}

\subjclass[2010]{Primary  37A35; Secondary 37A15}
\keywords{discrete spectrum, amenable group actions, almost periodic functions, bounded complexity, mean equicontinuous, equicontinuous in the mean}
\maketitle

\section{Introduction}

There is a long history for the study of discrete spectrum, which has attracted substantial interest in the literature since its introduction in dynamical systems.

Recently, discrete spectrum was studied in different settings by many mathematicians, for example,
in a wider context of the dynamics of metrics in measure spaces
by Vershik in a series of papers \cite{V10} and references therein, from the viewpoint of measure complexity by
Huang, Wang and Ye in \cite{HWY}
and Huang, Li, Thouvenot, Xu and Ye in \cite{HLTXY} and Vershik, Zatitskiy and Petrov in \cite{Vershik},
 flows by Scarpellini in \cite{S82}, measure dynamical systems
on locally compact $\sigma$-compact abelian groups and its equivalence to pure point diffraction by Lenz and Strungaru in \cite{LS09}, actions of locally compact $\sigma$-compact abelian groups on compact metric spaces by Lenz in \cite{Lenz}, actions of locally compact $\sigma$-compact amenable groups on compact metric spaces by Fuhrmann, Gr\"{o}ger and Lenz in \cite{FGL}.

\medskip

For a single measure-preserving transformation, an ergodic system with discrete spectrum may be the simplest ergodic system with zero measure-theoretic entropy.
 By the well-known Halmos-von Neumann Representation Theorem (see for example \cite{HVN} or
\cite[Theorem 3.6]{Walters}), the ``standard model" of an ergodic system (for a single transformation) with discrete spectrum is the minimal rotation over a compact abelian metric group.

Usually people describe the complexity of a dynamical system by
measuring the rate of the growth of the orbits in the system.
In \cite{Katok} Katok introduced the
complexity function via the modified notion of spanning sets
with respect to (w.r.t. for short) an invariant measure and an error parameter.
Recently, during their investigation of the Sarnak's conjecture, Huang, Wang and Ye \cite{HWY}
introduced the measure complexity of an invariant measure similar to the one introduced by Katok, by using the mean metric instead of the Bowen metric. They showed that if an
invariant measure has discrete spectrum, then the measure complexity w.r.t.
this invariant measure is bounded. Later in \cite{HLTXY} Huang, Li, Thouvenot, Xu and Ye proved that the converse statement remains true, and so
an
invariant measure has discrete spectrum if and only if it has bounded measure complexity.
Moreover, in \cite{HLTXY} the property of discrete spectrum was further connected to the so-called mean equicontinuity and equicontinuity in the mean.
Remark that the characterization of discrete spectrum using measure complexity was also proved by Vershik, Zatitskiy and Petrov in \cite{Vershik} via studying the class of so-called admissible metrics
in a Lebesgue space, under a mild condition that the standard Lebesgue space is non-atomic. In fact, this may go back to \cite{V11} via scaled entropy and even \cite{V10} by Vershik.
Different from the one by Katok, in \cite{Fer} Ferenczi characterized an
ergodic invariant measure with discrete spectrum
 via measure-theoretic
complexity using names of a partition and the Hamming distance, which was extended recently to non-ergodic case by Yu in \cite{Yu2018}.

\medskip

Note that for a topological dynamical system $(X, T)$, where $X$ is a compact metric space and $T: X\rightarrow X$ is a self-homeomorphism, there always exist
invariant Borel probability measures on $X$ so that the classic ergodic theory involves the study
of $(X, T)$.
Whereas, there are some groups $G$ such that there exists no any
invariant Borel probability measures on some compact metric space, where $G$ acts over it compatibly, for example
the rank two free group. It is well known that, for a group action with a countable discrete group, the
amenability of the group ensures the existence of invariant Borel probability measures, which
includes all finite groups, solvable groups and compact groups.

Thus the following question rises naturally. When considering group actions with a countable infinite discrete amenable group over compact metric spaces, may we still characterize discrete spectrum using measure complexity along the line of \cite{HLTXY, HWY, Vershik}, and is there still any relationship with the so-called mean equicontinuity, equicontinuity in the mean and measure-theoretic complexity as in \cite{Fer, HLTXY, Yu2018}?

\medskip

In this paper we will answer the above questions affirmatively.

\medskip

Recall that a countable infinite discrete group $G$ is {\it amenable} (see \cite{EF}) if there exists a
sequence of finite nonempty subsets $F_n\subset G$, which is called a {\it
F{\o}lner sequence}, such that $\lim_{n\rightarrow \infty} \frac{|g F_n\Delta F_n|}{|F_n|}=0$ for every
$g \in G$, where $|\bullet|$ denotes the cardinality of a subset $\bullet$.

By a \emph{topological dynamical system} (\emph{t.d.s.}\ for short)
we mean a pair $(X,G)$, where $X$ is a compact metric space with a metric $d$ and
$G$ is a countable infinite discrete amenable group acting as a group of homeomorphisms over $X$.
Denote by $\mathcal{B}_X$ the Borel $\sigma$-algebra of $X$, and $\mathcal{M} (X, G)$ the set of all invariant Borel probability measures over $X$, respectively. It is a standard fact that $\mathcal{M} (X, G)$ is a nonempty compact metric convex space, whose extreme points consist of exactly its ergodic elements. For more story of actions of amenable groups cf. \cite{MoulinOllagnier85book, OW, WZ}.
\emph{From now on, we will fix $(X, G)$ to be a t.d.s. with the metric $d$ and $\mu\in \mathcal{M} (X, G)$, and fix $\{F_n: n\in \mathbb{N}\}$ to be a F${\o}$lner sequence of $G$.}

Following \cite{Zimmer}, let $H_c\subset L^2(X,\mu)$ be the closed subspace generated by all finite dimensional
$G$-invariant subspaces of $L^2(X,\mu)$, and we say that $\mu$ has {\it discrete spectrum} if $H_c = L^2(X, \mu)$\footnote{Observe that, even if $\mu\in \mathcal{M} (X, G)$ has discrete spectrum, in general $L^2(X, \mu)$ is not necessarily spanned by
the eigenfunctions. Similar to \cite[Lemma 2.6]{Sakai}, if $L^2(X, \mu)$ is spanned by
the eigenfunctions then the group $G$ acts essentially as an abelian group over $(X, \mathcal{B}_X, \mu)$.}. According to \cite[Theorem 7.1]{Zimmer}, $H_c$ consists of all
\emph{almost periodic} functions $f\in L^2(X, \mu)$, that is, $\{f\circ g: g\in G\}$ is precompact in $L^2(X, \mu)$, where $f\circ g (x)= f (g x)$ for any $g\in G$ and all $x\in X$.
 In particular, $\mu$ has discrete spectrum
 if and only if any $f\in L^2(X,\mu)$ is almost periodic.

We show that $\mu$ has discrete spectrum if and only if it has bounded measure complexity, by characterizing almost periodic functions via the complexity function. This extends the corresponding results from an abelian group \cite{Vershik} (see also \cite{HLTXY, HWY, Yu2018} for the integer group $\mathbb{Z}$) to the case of a countable infinite discrete amenable group. We also prove that, discrete spectrum can be characterized via measure-theoretic complexity using names of a partition and the
Hamming distance, and it turns out to be equivalent to both mean equicontinuity
and equicontinuity in the mean. This generalizes the corresponding results \cite{Fer, HLTXY, Yu2018} from $\mathbb{Z}$-actions to our setting.

 \medskip

 In order to state precisely our main result, we need to make some preparations.

Let $w$ be a semimetric of $X$. Then the F${\o}$lner sequence $\{F_n: n\in \mathbb{N}\}$ generates naturally the sequence $\{\overline{w}_{F_n}: n\in \N\}$, where, for any $n\in \N$ and all $x, y\in X$,
$$\overline{w}_{F_n}(x,y)=\frac{1}{|F_n|}\underset{g\in F_n}{\sum}w(gx,gy).$$
For any $\ep>0$, we define {\it the complexity function of $\mu$ w.r.t. $\{\overline{w}_{F_n}: n\in \mathbb{N}\}$} as
$$C(\overline{w}_{F_n},\ep)=\min\left\{m\in\N: \exists x_1, \cdots, x_m\in X\ \text{s.t.}\ \mu\left(\bigcup_{i=1}^m B_{\overline{w}_{F_n}}(x_i,\ep/2)\right)>1-\ep\right\},$$
where $B_{\overline{w}_{F_n}}(x_i,\ep/2)= \{y\in X: \overline{w}_{F_n}(x_i,y)< \ep/2\}$.
We say that $\mu$ has \emph{bounded complexity w.r.t. $\{\overline{w}_{F_n}: n\in \mathbb{N}\}$} if for any $\ep>0$ there exists $M= M(\ep)\in \N$
such that $C(\overline{w}_{F_n},\ep)\le M$ for all $n\in \N$.
As we shall see, we are mainly interested in continuous semimetrics over $X$, we are also interested in the Hamming distance of names of a partition $\alpha$ (by a \emph{partition} $\alpha\subset \mathcal{B}_X$ we mean a finite family of pairwise disjoint subsets with $X$ as its union).

Denote by $\mathcal{P}_X$ the set of all partitions of $X$. Let $\alpha\in \mathcal{P}_X$ with $\alpha=\{A_1, \cdots, A_l\}, l\in \mathbb{N}$. For any point $x\in X$, the \emph{$\alpha$-name of $x$} is
$\{\alpha_g(x): g\in G\}$, where $\alpha_g(x)=k$ whenever $gx \in A_k$ for any $g\in G$.
It is trivial the following fact concerning the Hamming distance via $\alpha$-names that, for all $x,y\in X$ and each $n\in \mathbb{N}$,
 \[\frac{1}{|F_n|}|\{g\in F_n\colon \alpha_g(x)\neq \alpha_g(y)\}|= \overline{H^\alpha}_{F_n}(x,y)\]
where $H^\alpha(x,y)=\frac{1}{2}\sum_{i=1}^l |1_{A_i}(x)-1_{A_i}(y)|$.
It is also worth noticing that for $\alpha=\{A, X\setminus A\}$ one has $H^\alpha(x,y)=|1_{A}(x)-1_{A}(y)|$
for all $x,y\in X$.

We say that a compact set $K\subset X$ is \emph{mean equicontinuous} (\emph{equicontinuous in the mean}, respectively) w.r.t. $\{F_n: n\in \mathbb{N}\}$, if for every $\ep>0$,
there exists a $\delta>0$ such that $\limsup_{n\to\infty} \overline{d}_{F_n}(x,y)<\ep$ ($\overline{d}_{F_n}(x,y)<\ep$ for any $n\in \N$, respectively)
whenever $x,y\in K$ satisfy $d(x,y)<\delta$.
We say that $\mu\in \mathcal{M} (X,G)$ is \emph{mean equicontinuous} (\emph{equicontinuous in the mean}, respectively) w.r.t. $\{F_n: n\in \mathbb{N}\}$, if for every $\tau>0$
there exists a mean equicontinuous set (an equicontinuous in the mean set, respectively) $K\subset X$ w.r.t. $\{F_n: n\in \mathbb{N}\}$ such that $\mu(K)>1-\tau$. See \cite{Felipe, HLTXY} (and references therein) for history and recent progress about mean equicontinuity and equicontinuity in the mean.

\medskip

Recall that a sequence $\{E_n: n\in \mathbb{N}\}$ of finite subsets of $G$ is {\it tempered} if there exists
finite $C > 0$ such that  $|\underset{k<n}{\bigcup}E_k^{-1}E_n|\le C|E_n|$ for any $n\in \N$.
Every F\o lner sequence admits a
tempered subsequence \cite[Proposition 1.4]{Lindenstrauss}.

\medskip

The main result of this paper is stated as follows.

\begin{thm} \label{main}
 Let $(X,G)$ be a t.d.s. with metric $d$ and
$\mu \in M(X,G)$. Let $\{F_n: n\in \mathbb{N}\}$ be a F${\o}$lner sequence of $G$. Then the following statements are equivalent:
\begin{enumerate}
\item $\mu$ has discrete spectrum. \label{mfirst}

\item $\mu$ has bounded complexity w.r.t. $\{\overline{d}_{F_n}: n\in \mathbb{N}\}$. \label{msecond}

\item $\mu$ has bounded complexity w.r.t. $\{\overline{H^\alpha}_{F_n}: n\in \mathbb{N}\}$ for any $\alpha\in \mathcal{P}_X$. \label{mthird}

\item $\mu$ has bounded complexity w.r.t. $\{\overline{H^\alpha}_{F_n}: n\in \mathbb{N}\}$ for any $\alpha= \{A, X\setminus A\}\in \mathcal{P}_X$. \label{mfourth}
\end{enumerate}
Furthermore, if the F${\o}$lner sequence $\{F_n: n\in \mathbb{N}\}$ is tempered, then the above statements are equivalent to any of the following two statements:
\begin{enumerate}

\item[(5)] $\mu$ is mean equicontinuous w.r.t. $\{F_n: n\in \mathbb{N}\}$. \label{mfive}

\item[(6)] $\mu$ is equicontinuous in the mean w.r.t. $\{F_n: n\in \mathbb{N}\}$. \label{msix}
\end{enumerate}
\end{thm}

\medskip

The paper is organized as follows.
In section 2 we characterize almost periodic functions via
the complexity function. In section 3 we present the proof of Theorem \ref{main}.

\medskip

\medskip

\noindent {\bf Acknowledgments.}
We would like to thank Professors Xiangdong Ye, Weixiao Shen and Jian Li for useful suggestions.
TY is supported by China Postdoctoral Science Foundation, grant no. 2018M631996.
GHZ is supported by NSFC, grants no. 11671094, 11722103 and 11731003.
RFZ is supported by NSFC, grants no. 11671094 and 11871188.
This work was mainly carried out when TY was a Post
Doctoral Fellow at the Shanghai Center for Mathematical Sciences, Fudan University; the hospitality
of the Fudan University is greatly appreciated.

\section{Almost periodic functions and bounded complexity}

In this section we generalize \cite[Theorem 1.14]{GM15} and
 \cite[Theorem 5.2]{Yu2018} from $\Z$-actions to amenable group actions, by characterizing almost periodic functions via the complexity function, which will be crucial in our study of discrete spectrum in next section.

 \emph{We will
fix throughout the whole section the assumption that $h\in L^2(X,\mu)$ and $H(x, y)=|h(x)-h(y)|$ for all $x, y\in X$.}
The main result of this section is:

 \begin{thm} \label{201904291}
The function $h$ is almost periodic if and only if $\mu$ has bounded complexity w.r.t. $\{\overline{H}_{F_n}: n\in \mathbb{N}\}$.
\end{thm}

\medskip

Observe that
by the definition of bounded complexity (w.r.t. $\{\overline{w}_{F_n}: n\in \mathbb{N}\}$), one only concludes that there exists a uniform bound $M\in \mathbb{N}$ (depending on the given parameter $\ep> 0$ and independent of $n\in \N$), such that for all $n\in \N$ there exist at most $M$ subsets covering the whole space $X$ (up to an $\ep$ error) and satisfying that each such subset has its $\overline{H}_{F_n}$-diameter smaller than $\ep$. We don't know if we can choose these $M$ subsets independent of each step $n\in \mathbb{N}$.
Lindenstrauss proved pointwise ergodic convergence theorem along averages over tempered F\o lner sequence \cite[Theorem 1.3]{Lindenstrauss}. With the help of this, we have the following result, which will be useful in later discussions.

\begin{lem}\label{20190324}
Assume that the F${\o}$lner sequence $\{F_n: n\in \mathbb{N}\}$ is tempered.
  Then $\mu$ has bounded  complexity w.r.t. $\{\overline{H}_{F_n}: n\in \mathbb{N}\}$, if and only if, for any $1> \ep > 0$, there exist finitely many Borel subsets $X_1, \cdots, X_s$
such that $\mu(\bigcup_{j=1}^s X_j)>1-\ep$ and
   $\overline{H}_{F_n}(x,y)<\ep$ for all $x,y\in X_j$ (for some $j= 1, \cdots, s$) and any $n\in \N$.
  \end{lem}
\begin{proof}
It suffices to prove the direction of ``$\Rightarrow$". Let $1> \ep>0$ and fix it.

By the assumption that $\mu$ has bounded  complexity w.r.t. $\{\overline{H}_{F_n}: n\in \mathbb{N}\}$, there is $C=C(\ep)\in \mathbb{N}$ such that for any $n\in\N$,
there is finite $E_n\subset X$ with $|E_n|\leq C$  such that
\[
\mu\biggl(\bigcup_{x\in E_n} B_{\overline{H}_{F_n}}(x,\ep/8)\biggr)>1-\ep/8.
\]
Since the F${\o}$lner sequence $\{F_n: n\in \mathbb{N}\}$ is tempered, applying the pointwise ergodic convergence theorem \cite{Lindenstrauss} one has that the sequence $\{\overline{H}_{F_n}(x,y): n\in \mathbb{N}\}$ converges for $\mu\times \mu$-a.e. $(x,y)\in X\times X$. Thus, for a given $0<r< \frac{\ep}{4C}$, by Egorov theorem there exists compact $R\subset X\times X$ with $\mu\times \mu(R)>1-r^2$ and $N\in\N$
such that if $(x,y)\in R$ then
$$|\overline{H}_{F_n}(x,y)-\overline{H}_{F_N}(x,y)|<r\ \text{for all}\ n\ge N.$$
And then by Fubini theorem there is a Borel set $A\subset X$ such that $\mu(A)>1-r$ and
$\mu(R_x)>1-r$ for any $x\in A$, where $R_x=\{y\in X: (x,y)\in R\}$.
Let us enumerate $E_{N}=\{z_1, \cdots, z_m\}$ (clearly $m\leq C$), and
set \[I=\left\{1\leq i\leq m\colon A\cap B_{\overline{H}_{F_N}}(z_i, \ep/8)
\neq\emptyset\right\}\ \text{and}\ m'= |I|\le m.\]
For each $i\in I$, pick $y_i\in A\cap  B_{\overline{H}_{F_N}}(z_i, \ep/8)$, and hence $B_{\overline{H}_{F_N}}(z_i, \ep/8)\subset
B_{\overline{H}_{F_N}}(y_i, \ep/4)$.
As
\begin{align*}
\mu\biggl( A \cap \bigcap_{i\in I} R_{y_i}
\cap \bigcup_{x\in E_N} B_{\overline{H}_{F_N}}(x, \ep/8)
\biggr)\geq 1-r-m'r-\ep/8>1-\ep,
\end{align*}
by Lusin Theorem we can choose for any $j\in I$ a compact set
\[K_j \subset A \cap \bigcap_{i\in I} R_{y_i}
\cap B_{\overline{H}_{F_N}}(y_j, \ep/4)\]
such that $\mu(K)>1-\ep$ and $h (g x)$ is continuous on $K$ for all $g\in \bigcup_{n=1}^{N-1} F_n$, where $K=\underset{j\in I}{\bigcup} K_j$.

Since $h (g x)$ is continuous on $K$ for all $g\in \bigcup_{n=1}^{N-1} F_n$, there is $\delta>0$ such that $|h(gx_1)-h(gx_2)|<\ep$ for all $g\in \bigcup_{n=1}^{N-1} F_n$ (and hence $\overline{H}_{F_n}(x_1,x_2)<\ep$ for all $1\leq n \leq N-1$), whenever $x_1, x_2\in K$ satisfy $d(x_1,x_2)<\delta$.
Note that for any $j\in I$ the subset $K_j$ is compact, there exists finite $H_j\subset K_j$ such that $K_j\subset \underset{x\in H_j}{\bigcup}B(x,\frac{\delta}{2})$, where $B(x,\frac{\delta}{2})= \{y\in X: d (x, y)< \frac{\delta}{2}\}$.

Now we will prove that
$\left\{B(x,\frac{\delta}{2})\cap K_i: i\in I\ \text{and}\ x\in H_i\right\}$ satisfies the required property. Clearly $K= \underset{i\in I}{\bigcup} \left(\underset{x\in H_i}{\bigcup} B(x,\frac{\delta}{2})\cap K_i\right)$, whose $\mu$-measure is strictly larger than $1- \ep$.
Now assume $x_1,x_2\in B(x,\frac{\delta}{2})\cap K_i$ for some $x\in H_i$ and $i\in I$.
By the selection of $\delta$ one obtains $\overline{H}_{F_n}(x_1,x_2)< \ep$ for all $1\leq n \leq N-1$.
Since for each $k= 1, 2$, $x_k\in K_i\subset R_{y_i}\cap B_{\overline{H}_{F_N}}(y_i, \ep/4)$,
one has $(y_i, x_k)\in R$ and then, by the construction of $R$,
\[\overline{H}_{F_n}(y_i,x_k)\le
\overline{H}_{F_N}(y_i,x_k)+r<\ep/4+r<\ep/2\ \ \text{for any $n\ge N$}.\]
Now we have $\overline{H}_{F_n}(x_1,x_2)\leq \overline{H}_{F_n}(x_1,y_i)+\overline{H}_{F_n}(x_2,y_i)<\ep$ for any $n\geq N$ (and hence for all $n\geq 1$), which ends the proof.
\end{proof}

We also need the following technical results.

\begin{lem}\label{20190321}
Let $f\in A$, where $A$ is a finite dimensional $G$-invariant subspace of $L^2(X,\mu)$
with an orthonormal basis $\{e_1,\cdots,e_k\}$ (under $\| \bullet \|_2$ inherited from $L^2(X,\mu)$).
Then
for $\mu\times \mu$-a.e. $(x,y)\in X\times X$ one has, for any $g\in G$,
$$|f(gx)-f(gy)|\le \| f\|_2\cdot \sum_{i=1}^k |e_i(x)-e_i(y)|.$$
\end{lem}
\begin{proof}
Since $A$ is $G$-invariant, for any $g\in G$ there exist $a_g^i\in \mathbb{C},i=1,\cdots,k$
such that $f\circ g= \sum_{i=1}^k a_g^i e_i$ in $L^2(X,\mu)$, and then $f\circ g (x)= \sum_{i=1}^k a_g^i e_i (x)$ for $\mu$-a.e. $x\in X$,
hence $\| f\|_2^2=\| f\circ g\|_2^2=\sum_{i=1}^k |a_g^i|^2$ (and so $|a_g^i|\le \| f\|_2$ for any $g\in G$ and all $i=1,\cdots,k$). As $G$ is a countable group, the conclusion follows directly.
\end{proof}

\begin{lem}\label{20190325}
Assume that $|h|\le C$ for some $C>0$ and $\int_X h \ d\mu=0$.
If $\int_{X\times X} H \ d\mu\times\mu< \frac{1}{k^2}$ for some $k> 1$,
then $\int_X |h| \ d\mu\le \frac{2+C}{k}+\frac{C}{k-1}$.
\end{lem}
\begin{proof}
By the assumption $\mu\times\mu(A)<\frac{1}{k}$, where $A=\{(x,y)\in X\times X: H(x,y)\ge \frac{1}{k}\}$.
Then there exists $y\in X$ such that $\mu(A_y)<\frac{1}{k}$ where $A_y=\{x\in X:(x,y)\in A\}$. And hence for $d=h(y)$, one has that if $x\in X\setminus A_y$ then $|h(x)-d|<\frac{1}{k}$, thus
\[\int_{X\setminus A_y} \left(d-\frac{1}{k}\right)\ d\mu\le \int_{X\setminus A_y} h\ d\mu\le
\int_{X\setminus A_y} \left(d+\frac{1}{k}\right)\ d\mu.\] This implies from $|h|\le C$ that
 $$\int_{X\setminus A_y} \left(d-\frac{1}{k}\right) d\mu - C\mu(A_y)\le \int_{X} h\ d\mu=0\le
\int_{X\setminus A_y} \left(d+\frac{1}{k}\right) d\mu+ C\mu(A_y).$$
On one hand, either $d\ge - \frac{1}{k}$, or $d< - \frac{1}{k}$ which implies $(1-\frac{1}{k})(d+\frac{1}{k})\ge - \frac{C}{k}$ and hence $d\ge - \frac{C}{k-1}- \frac{1}{k}$. On the other hand, either $d< \frac{1}{k}$, or $d\ge \frac{1}{k}$ which implies $(1-\frac{1}{k})(d-\frac{1}{k})\le \frac{C}{k}$ and finally $d\le \frac{C}{k-1}+\frac{1}{k}$. Summing up, $|d|\le \frac{C}{k-1}+\frac{1}{k}$ and then
$\int_X |h| \ d\mu\le \int_{X\setminus A_y} (|h-d|+|d|) \ d\mu+C\mu(A_y)\le  \frac{2+C}{k}+\frac{C}{k-1}.$
\end{proof}

\begin{lem}\label{20190319}
Set $L (x,y)=h(x)-h(y)$ for all $x, y\in X$ and assume $|h|\le C$ for some $C>0$.
If $h$ is not almost periodic, then
 $L$ is not almost periodic.
\end{lem}
\begin{proof}
Set $L_g(x,y)=L(gx,gy)=h(gx)-h(gy)$ for all $x, y\in X$ and any $g\in G$.
 Assume the contrary that $L$ is almost periodic in $L^2(X\times X,\mu\times \mu)$,
then for any $k> 1$ there exists $\{g_{k, 1},\cdots, g_{k, l_k}\}\subset G$ such that
for any $g\in G$, there exists $1\le m\le l_k$
with $$\frac{1}{k^2}> \left(\int_{X\times X}|L_{g_{k, m}}-L_g|^2 d\mu\times\mu\right)^{\frac{1}{2}}\ge \int_{X\times X}|L_{g_{k, m}}-L_g| d\mu\times\mu.$$
Since $\int_X (h\circ g_{k, m}-h\circ g) \ d\mu=0$ and $|h|\le C$ (and hence $|h\circ g_{k,m}-h\circ g|\le 2C$),
one has $$\int_X |h\circ g_{k,m}-h\circ g| \ d\mu\le
\frac{2+2C}{k}+\frac{2C}{k-1}$$
from Lemma \ref{20190325}.
So $$\int_X |h\circ g_{k,m}-h\circ g|^2 \ d\mu\le 2C\int_X |h\circ g_{k,m}-h\circ g| \ d\mu
\le\frac{4C+4C^2}{k}+\frac{4C^2}{k-1}.$$
By arbitrariness of $k> 1$ we have that $h$ is almost periodic, a contradiction.
\end{proof}

Before proving Theorem \ref{201904291}, let us show firstly the following intermediate result.

\begin{thm} \label{20190320}
Assume that $h$ is almost periodic and that the F${\o}$lner sequence $\{F_n: n\in \mathbb{N}\}$ is tempered.
  Then for any $\ep > 0$, there exists $\delta>0$ and compact $B\subset X$
  with $\mu(B)>1-\ep$ such that
   $\overline{H}_{F_n}(x,y)<\ep$ for each $n\in \N$ and all $x,y\in B$ with $d(x,y)<\delta$.
In particular, $\mu$ has bounded complexity w.r.t. $\{\overline{H}_{F_n}: n\in \mathbb{N}\}$.
\end{thm}
\begin{proof}
Without loss of generality we may assume that $h$ is not the zero function.

Take any $\ep> 0$ and fix it.
Let $A\subset L^2(X,\mu)$ be the closed subspace generated by all finite dimensional
$G$-invariant subspaces of $L^2(X,\mu)$, which consists of all
almost periodic functions by \cite[Theorem 7.1]{Zimmer}.
Then we can decompose $A$ as $\bigoplus_{i\in I} A_i$, where $I$ is a countable nonempty index set and for each $i\in I$ the space
$A_i$ is a finite dimensional $G$-invariant subspace of $L^2(X,\mu)$
with an orthonormal basis $\{e_{i,1}, \cdots,e_{i,l_i}\}$, and hence
for each $i\in I$ there exists $f_{i}\in A_i$ such that
$h=\sum_{i\in I} f_{i}$. Thus there exists a finite nonempty set $K\subset I$ such that $f_i\in A_i\setminus \{0\}$ for each $i\in K$ and
$$\int_X \left| h-\sum_{i\in K}  f_{i}\right|\ d\mu\le \left\| h-\sum_{i\in K} f_{i}\right\|_2<\frac{\ep^2}{25}.$$

Since the F${\o}$lner sequence $\{F_n: n\in \mathbb{N}\}$ is tempered,
by the pointwise ergodic convergence theorem \cite{Lindenstrauss}, there exists $X'\in \mathcal{B}_X$ with $\mu(X')=1$ and $h^*\in L^1(X,\mu)$ such that
$$\lim_{n\ra \infty}\frac{1}{|F_n|}\sum_{g\in F_n}\left|h(gx)-\sum_{i\in K}  f_{i}(gx)\right|= h^*(x)$$
for each $x\in X'$ and $\int_X h^* d\mu<\frac{\ep^2}{25}$. Then
$\mu(Y)>1-\frac{\ep}{5}$, where $Y=\{x\in X': h^*(x)< \frac{\ep}{5}\}$.
By Egorov theorem, there exists $Y'\subset Y$ with $\mu(Y')>1-\frac{\ep}{5}$ and $N\in \N$
such that $$\left|\frac{1}{|F_n|}\sum_{g\in F_n} \big|h(gx)-\sum_{i\in K}  f_{i}(gx)\big|-h^*(x)\right|<\frac{\ep}{5}$$
for all $x\in Y'$ and each $n\ge N$. Hence by above construction,
for all $x\in Y'$ and each $n\ge N$,
\[\frac{1}{|F_n|}\sum_{g\in F_n}\left|h(gx)-\sum_{i\in K}  f_{i}(gx)\right|<\frac{2\ep}{5}.\]
And then, for all $x,y\in Y'$ and each $n\ge N$,
\begin{eqnarray*}
& & \frac{1}{|F_n|}\sum_{g\in F_n}|h(gx)-h(gy)|\\
& \leq & \frac{1}{|F_n|}\sum_{g\in F_n} \left(\left|h(gx)-\sum_{i\in K}  f_{i}(gx)\right|
   + \left|h(gy)-\sum_{i\in K}  f_{i}(gy)\right| + \left|\sum_{i\in K}  (f_{i}(gx) -f_{i}(gy))\right|\right)\\
& \leq & \frac{2\ep}{5}+\frac{2\ep}{5}+\frac{1}{|F_n|}\sum_{g\in F_n}
\sum_{i\in K} |f_{i}(gx) -f_{i}(gy)|.
\end{eqnarray*}
By Lemma \ref{20190321}, for any $i\in K$ and $\mu\times \mu$-a.e. $(x,y)\in X\times X$, it holds that
$|f_{i}(gx) -f_{i}(gy)|\le \| f_i\|_2\cdot \sum_{k=1}^{l_i} |e_{i,k}(x)-e_{i,k}(y)|$ for any $g\in G$ (without loss of generality we may assume that it holds for each $g\in G$ and all $x, y\in Y'$). Thus, for all $x,y\in Y'$ and each $n\ge N$,
$$\frac{1}{|F_n|}\sum_{g\in F_n}|h(gx)-h(gy)|\le
\frac{4\ep}{5}+\sum_{i\in K} \sum_{k=1}^{l_i} \| f_i\|_2\cdot |e_{i,k}(x)-e_{i,k}(y)|.$$

Now applying Lusin theorem, we may choose compact $B\subset Y'$ such that
$\mu(B)>1-\ep$ and
all of $e_{i,k}, i\in K, 1\le k\le l_i$ and $h\circ g, g\in \bigcup_{n=1}^{N-1} F_n$ are continuous on $B$, and then there exists $\delta>0$ such that if $x,y\in B$ satisfy
$d(x,y)<\delta$ then $|e_{i,k}(x)-e_{i,k}(y)|<\frac{\ep}{5\sum_{i\in K} \| f_i\|_2\cdot l_i}$
and $|h(gx)-h(gy)|<\ep$
whenever $i\in K, 1\le k\le l_i$ and $g\in \bigcup_{n=1}^{N-1} F_n$.
Thus we have
$$\frac{1}{|F_n|}\sum_{g\in F_n}|h(gx)-h(gy)|< \ep$$
whenever $n\in \N$ and $x,y\in B$ satisfying
$d(x,y)<\delta$.

Finally, since $B$ is compact, there exist $x_1,\cdots,x_s\in B, s\in \N$ such that
$B\subset \underset{j=1}{\overset{s}{\bigcup}}B(x_j,\frac{\delta}{2})$.
Let $X_0=X\setminus B$, $X_j=B(x_j,\frac{\delta}{2})\cap B$ for each $j= 1, \cdots, s$. Thus
$\mu(X_0)< \ep$ and $\overline{H}_{F_n}(x,y)<\ep$ for all $n\in \N$ once both $x$ and $y$ come from the same $X_j$ ($j= 1, \cdots, s$),
which implies that $C(\overline{H}_{F_n},\ep)\le s$ for all $n\in \N$ and hence
$\mu$ has bounded  complexity w.r.t. $\{\overline{H}_{F_n}: n\in \mathbb{N}\}$.
\end{proof}


The following theorem is inspired by \cite[Theorem 6 and Theorem 7]{Vershik}.

\begin{thm} \label{0905}
If $\mu$ has bounded  complexity w.r.t. $\{\overline{H}_{F_n}: n\in \mathbb{N}\}$, then $h$ is almost periodic.
\end{thm}
\begin{proof}
Remark that $\lim_{r\ra + \infty} \|h^r-h\|_2= 0$, where for each $r>0$ we construct  in $L^2(X,\mu)$
$$h^r(x)=\begin{cases}
       h(x),& |h(x)|\leq r\\
       \ \\
       r\cdot \frac{h(x)}{|h(x)|},& |h(x)|> r
       \end{cases}\ \ \ \text{for each}\ x\in X.
$$
Assume the contrary that $h$ is not almost periodic, by \cite[Theorem 7.1]{Zimmer}
there exists $R>0$ such that $h^R$ is not almost periodic.
We set, for all $x, y\in X$, $L(x,y)=h(x)-h(y)$ and
$L^R(x,y)=h^R(x)-h^R(y)$, $H^R(x,y)= |L^R(x,y)|$, which implies
\begin{equation} \label{20190718}
|h(x)-h(y)|\geq |h^R(x)-h^R(y)|\ \ \text{and hence}\ \ \overline{H}_{F_n}(x,y)\ge \overline{H^R}_{F_n}(x,y).
\end{equation}
Since  $\mu$ has bounded  complexity w.r.t. $\{\overline{H}_{F_n}: n\in \N\}$,
it is easy to obtain from \eqref{20190718} that $\mu$ has bounded  complexity w.r.t. $\{\overline{H^R}_{F_n}: n\in \N\}$.
Now let $h'(x)=\frac{h^R(x)}{4R}$ for all $x\in X$. Then $h'\in L^2(X,\mu)$ and $|h'|\leq \frac{1}{4}$. It is easy to see that $h'\in L^2(X, \mu)$ is not almost periodic, and so we may write alternatively
$h$ instead of $h'$ in the following.

Observe from Lemma \ref{20190319} that $L$ is not almost periodic in $L^2(X\times X,\mu\times \mu)$, and so there exists $c_1> 0$ and $g_m\in G, m\in \N$
such that $\|L_{g_i}-L_{g_j}\|_2\geq c_1$
for all $i\not=j \in \N$ (where $L (x, y)= h (x)- h (y)$ and $L_g (x, y)= L (g x, g y)$ for all $x, y\in X$ and each $g\in G$).
Note that
$|L|\leq \frac{1}{2}$, and then
$|L_{g_i}-L_{g_j}|\le 1$, hence, for any $i\not=j \in \N$,
\begin{equation} \label{night2}
\int_{X \times X} |L_{g_i}-L_{g_j}| d\mu\times \mu \geq \int_{X \times X} |L_{g_i}-L_{g_j}|^2 d\mu\times \mu \ge (c_1)^2\ (\text{denoted by $c$}).
\end{equation}

Without loss of generality, we may assume that the F${\o}$lner sequence $\{F_n: n\in \N\}$ is tempered.
Let $\ep=\frac{c}{24}$. Then by Lemma \ref{20190324}, there exist nonempty Borel sets $X_1, \cdots, X_k$
  with $\mu(\bigcup_{i=1}^k X_i)>1-\ep$ such that
   $\overline{H}_{F_n}(x,y)<\ep$ for all $x,y\in X_i$ (for some $i= 1, \cdots, k$) and any $n\in \N$.
We choose $p_i\in X_i$ arbitrarily for each $i = 1, \cdots, k$.
For any $g\in G$ and all $x, y\in X$, we denote $L_g (x,y)=L(gx,gy), H_g(x,y)=H(gx,gy)$ and define $f_g(x,y)$ as
$$
f_g(x,y)=\begin{cases}
       0,& x\in X_0\, \text{ or }\, y\in X_0,\ \text{where}\ X_0= X\setminus \bigcup_{i= 1}^k X_i\\
       \ \\
       L_g(p_i,p_j),& x\in X_i\ \text{and}\ y\in X_j\quad (i, j= 1, \cdots, k)
       \end{cases}.
$$

Firstly we will estimate the sum of the $L^1$-distances between the pairs of functions $f_g$ and $L_g$ over $g\in F_n$ as follows.
If we assume $(x,y)\in X_i\times X_j$, since
\begin{equation*}
|L_g(x, y)- f_g(x, y)| = |L_g(x, y)- L_g(p_i, p_j)| \leq H_g(x, p_i) + H_g(y, p_j),
\end{equation*}
 summing these inequalities over $g\in F_n$ we obtain in the right-hand side that
\begin{equation} \label{201907182301}
\underset{ g\in F_n}{\sum} H_g(x, p_i) = |F_n|\cdot \overline{H}_{F_n}(x, p_i)<|F_n|\ep, \
\underset{ g\in F_n}{\sum} H_g(y, p_j) = |F_n|\cdot \overline{H}_{F_n}(y, p_j)<|F_n|\ep.
\end{equation}
 Let $X_{00}=(X_0\times X)\cup (X \times X_0)$. Then $\mu(X_{00})<2\ep$, and hence by the construction
\begin{eqnarray*}
& &\sum_{g\in F_n} \int_{X\times X}|L_g-f_g|d\mu\times\mu \\
&= & \overset{k}{\underset{ i,j=1}{\sum}}\sum_{g\in F_n}\int_{X_i\times X_j}|L_g-f_g|d\mu\times\mu
+\sum_{g\in F_n}\int_{X_{00}}|L_g-f_g|d\mu\times\mu \\
&\le & \overset{k}{\underset{ i,j=1}{\sum}}\sum_{g\in F_n}\int_{X_i\times X_j} [H_g(x, p_i) + H_g(y, p_j)]d\mu\times\mu +\sum_{g\in F_n}\int_{X_{00}}H_gd\mu\times\mu \\
&\le & \overset{k}{\underset{ i,j=1}{\sum}}\int_{X_i\times X_j} (|F_n| \ep+ |F_n| \ep)d\mu\times\mu + \sum_{g\in F_n} \int_{X_{00}}H_gd\mu\times\mu\ (\text{using \eqref{201907182301}}) \\
& <& 4|F_n|\ep\ (\text{observe that $H_g\le 1$ for all $g\in G$}).
\end{eqnarray*}
Since $4\ep =\frac{c}{6}$, there exists $F_n''\subset F_n$ with $|F_n''|\geq \frac{|F_n|}{2}$
such that, for all $g\in F_n''$,
\begin{equation} \label{night}
\int_{X\times X} |L_g-f_g| d \mu\times \mu < \frac{c}{3}.
\end{equation}

Now we consider $A=\big\{w\in L^1(X\times X, \mu\times \mu): w=\overset{k}{\underset{ i,j=1}{\sum}}c_{ij}1_{X_i\times X_j}\ \text{with all of}\ |c_{ij}|\leq \frac{1}{2}\big\}$. Clearly $\{f_g: g\in G\}\subset A$. Since $A$ has dimension $k^2$, there exists $C\in \N$ and
 $F_n'\subset F_n''$ with
 $|F_n'| \geq\frac{|F_n''|}{C}  \geq\frac{|F_n|}{2C}$ such that, for all $u,v\in F_n'$, we have
$\int_{X\times X} |f_{u}-f_{v}| d \mu\times \mu\leq \frac{c}{3}$,
and hence
\[\int_{X\times X} |L_{u}-L_{v}| d \mu\times \mu< c\ (\text{by \eqref{night}}).\]
Set $M=5C$. Since $\{F_n: n\in \N\}$ is a F${\o}$lner sequence of $G$, once $n\in \N$ is large enough, for all $m=1,\cdots,M$ one has
$\frac{|g_m^{-1}F_n\Delta F_n|}{|F_n|}<\frac{1}{M}$, and then
$$\frac{| g_m^{-1}F_n'\setminus F_n|}{|F_n|}\le \frac{|g_m^{-1}F_n\setminus F_n|}{|F_n|}<\frac{1}{M}\ \text{(recalling $F_n'\subset F_n''\subset F_n$)},$$
hence
$$\frac{| g_m^{-1}F_n'\cap F_n|}{|F_n|}= \frac{|F_n'|}{|F_n|}-\frac{| g_m^{-1}F_n'\setminus F_n|}{|F_n|}> \frac{1}{2C}-\frac{1}{M}=\frac{3}{10C}.$$
Thus, once $n\in \N$ is large enough,
one has
 $$\overset{M}{\underset{m=1}{\sum}}|g_m^{-1}F_n'\cap F_n|>\frac{3|F_n|}{10C}\times M=\frac{3}{2}|F_n|,$$
and so there exist $1 \leq p\not = q\leq M$ such that $g_p^{-1}F_n'\cap g_q^{-1}F_n'\cap F_n\not=\emptyset$.

We choose $w \in g_p^{-1}F_n'\cap g_q^{-1}F_n'\cap F_n$, and say $u,v\in F_n'$ such that
$g_pw=u$ and $g_qw=v$.
 Then by above construction of \eqref{night2} and $F_n'$, we have $\int_{X\times X} |L_{g_p}-L_{g_q}| d \mu\times \mu\geq c$ and $\int_{X\times X} |L_{u}-L_{v}| d \mu\times \mu <c$, while
\[\int_{X\times X} |L_{g_p}-L_{g_q}| d \mu\times \mu= \int_{X\times X} |L_{g_p w}-L_{g_q w}| d \mu\times \mu
=\int_{X\times X} |L_{u}-L_{v}| d \mu\times \mu,\]
which arrives at a contradiction. Thus the function $h$ is almost periodic.
\end{proof}

Now it is ready to prove Theorem \ref{201904291}.

\begin{proof}[Proof of Theorem \ref{201904291}]
By Theorem \ref{0905} it suffices to prove that if $h$ is almost periodic then $\mu$ has bounded  complexity w.r.t. $\{\overline{H}_{F_n}: n\in \mathbb{N}\}$. Assume the contrary that
there exists $\ep>0$ such that the sequence $\{C(\overline{H}_{F_n},\ep): n\in \N\}$ is not bounded,
and so there exists a tempered subsequence $\{F_{n_k}: k\in \N\}$ such that the sequence $\{C(\overline{H}_{F_{n_k}},\ep): k\in \N\}$
is not bounded yet. While since $\{F_{n_k}: k\in \N\}$ is a tempered F${\o}$lner sequence, by Theorem \ref{20190320} one has that the sequence
$\{(\overline{H}_{F_{n_k}},\ep): k\in \N\}$ is bounded, a contradiction.
\end{proof}

\section{Proof of Theorem \ref{main}}

In this section, we present the proof of Theorem \ref{main}.

\medskip

The equivalence of $\eqref{mfirst}\Leftrightarrow \eqref{msecond}$ in Theorem \ref{main} follows from the following result, which generalizes \cite[Proposition 4.1]{HWY},
 \cite[Theorem 4.6]{HLTXY} and corresponding results in \cite{Vershik}
 from $\Z$-actions to amenable group actions.
Recall that we have assumed $X$ to be a compact metric space with the metric $d$. By
 a {\it semimetric} of $X$ we mean
$\rho: X\times X \ra \mathbb{R}_+$ satisfying that
$\rho(x,x)=0$,
$\rho(x,y)=\rho(y,x)$ and $\rho(x,z)\le \rho(x,y)+\rho(y,z)$ for all $x, y, z\in X$.
We say that the semimetric $\rho$ is \emph{continuous}, if the identity map
$Id: (X,d)\ra (X,\rho)$ is continuous.
Observe that if $\rho$ is a continuous metric on $X$,
then $Id: (X,d)\ra (X,\rho)$ will be a homeomorphism, that is,
the topology induced by $\rho$ coincides with that induced by $d$.

\begin{thm}\label{20190427}
The following statements are equivalent:
\begin{enumerate}
\item  $\mu$ has discrete spectrum. \label{0905-01}

\item For  some continuous metric $d'$, $\mu$ has bounded  complexity w.r.t. $\{\overline{d'}_{F_n}: n\in \N\}$. \label{0905-02}

    \item \label{0905-03} For any continuous semimetric $\rho$,
$\mu$ has bounded  complexity w.r.t. $\{\overline{\rho}_{F_n}: n\in \N\}$.

\item For the metric $d$, $\mu$ has bounded complexity w.r.t. $\{\overline{d}_{F_n}: n\in \N\}$. \label{0905-04}
\end{enumerate}
\end{thm}
\begin{proof}
It is trivial
$\eqref{0905-03}\Rightarrow \eqref{0905-04}\Rightarrow \eqref{0905-02}$, and so it remains to prove
$\eqref{0905-01}\Longleftrightarrow \eqref{0905-02}\Longleftrightarrow \eqref{0905-03}$.

$\eqref{0905-01}\Rightarrow \eqref{0905-02}$
Choose a countable set $\{f_k: k\in \N\}\subset C(X)$
such that $\underset{x\in X}{\max} |f_k(x)| \leq 1$ for each $k\in \N$ and the linear span of $\{f_k: k\in \N\}$ is dense in $C(X)$.
Then $d'$ is a compatible metric on $X$, equivalently, $d'$ is a continuous metric on $X$, where
$$d'(x,y)=\sum_{k=1}^\infty\frac{|f_k(x)-f_k(y)|}{2^k}\ \text{for all}\ x, y\in X.$$
In the following we will show that $\mu$ has bounded  complexity w.r.t. $\{\overline{d'}_{F_n}: n\in \N\}$, that is, for each $\ep> 0$ the sequence $\{C(\overline{d'}_{F_n},\ep): n\in \N\}$ is bounded.

Fix $\ep> 0$. Obviously we may select $N\in \N$ such that
$$d'(x,y)\le \sum_{k=1}^N\frac{|f_k(x)-f_k(y)|}{2^k}+ \frac{\ep}{2}\ \text{for all $x, y\in X$}.$$

We assume firstly that the F${\o}$lner sequence $\{F_n: n\in \N\}$ is tempered.
Since $\mu$ has discrete spectrum, each $f_k$ is almost periodic, and so by Theorem \ref{20190320} there exists $\delta_k> 0$ and compact
$A_k\subset X$ with $\mu(A_k)>1-\frac{\ep}{2^k}$ such that
$\frac{1}{|F_n|}\underset{ g\in F_n}{\sum}|f_k(gx)-f_k(gy)|<\frac{\ep}{2}$
 for each
$n\in \N$ and all $x,y\in A_k$ with $d(x,y)<\delta_k$.
Set $A=\underset{k=1}{\overset{N}{\bigcap}} A_k$ and $\delta=\underset{1\le k\le N}{\min} \delta_k> 0$.
Then $A\subset X$ is compact, $\mu (A)> 1- \ep$ and, once $x,y\in A$ satisfy $d(x,y)<\delta$, one has
\begin{equation*}
\overline{d'}_{F_n}(x,y)\le \sum_{k=1}^N \frac{1}{|F_n|}\underset{g\in F_n}{\sum} \frac{|f_k(gx)-f_k(gy)|}{2^k} +\frac{\ep}{2}<\ep.
\end{equation*}
From this it is easy to obtain that the sequence $\{C(\overline{d'}_{F_n},\ep): n\in \N\}$ is bounded.

Now we consider a general F${\o}$lner sequence $\{F_n: n\in \N\}$. Assume the contrary that the sequence $\{C(\overline{d'}_{F_n},\ep): n\in \N\}$ is not bounded,
and then by arguments similar to the proof of Theorem \ref{201904291}
we will arrive at a contradiction.

$\eqref{0905-02}\Rightarrow \eqref{0905-03}$
Let $\rho$ be a continuous semimetric on $X$. Since $(X, d)$ is a compact metric space, and the identity map
$Id: (X,d)\ra (X,\rho)$ is continuous, it is easy to see that the function $\rho$ is bounded, and so we may assume without loss of generality $\rho\leq 1$.

Now assume that
$d'$ is a continuous metric such that $\mu$ has bounded  complexity w.r.t. $\{\overline{d'}_{F_n}: n\in \N\}$. Let $\ep> 0$ and fix it.
Observe that there exists $\delta>0$ such that, for all $x, y\in X$, $d'(x,y)<\delta$ implies $\rho(x,y)<\ep$.
By the assumption of bounded complexity, there exists $k\in \N$ such that, for each $n\in \N$,
$X$ can be represented as the union of Borel sets $X_0^n, X_1^n,\cdots,X_k^n$, furthermore,
$\mu(X_0^n) <\ep\delta$ and $\overline{d'}_{F_n} (x,y)<\ep\delta$ whenever $x,y\in X_j^n$ for some $j = 1,\cdots,k$.
In the following we will prove that $\overline{\rho}_{F_n}(x,y)< 2\ep$ whenever $x,y\in X_j^n$ for some $n\in \N$ and $j = 1,\cdots,k$, which implies $C(\overline{\rho}_{F_n}, 2 \ep)\le k$ for all $n\in \N$ and then
$\mu$ has bounded  complexity w.r.t. $\{\overline{\rho}_{F_n}: n\in \N\}$ by the arbitrariness of $\ep> 0$.

Now fix $x,y\in X_j^n$, where $n\in \N$ and $j = 1,\cdots,k$. Since $\overline{d'}_{F_n} (x,y)<\ep\delta$, one has
\[|E_n|\geq  |F_n|(1-\ep),\ \text{where}\ E_n=\{g\in F_n: d'(gx,gy)< \delta\}.\]
By the construction of $\delta$, we have $\rho(gx,gy)<\ep$ for all $g\in E_n$, and then
$$\overline{\rho}_{F_n}(x,y)\leq \frac{1}{|F_n|}\cdot (|E_n|\ep+|F_n\setminus E_n|)\ (\text{recall $\rho\le 1$})\ <2\ep.$$

$\eqref{0905-03}\Rightarrow \eqref{0905-01}$ Let $h$ be arbitrary continuous function on $X$, and set $H(x,y)=|h(x)-h(y)|$ for all $x, y\in X$. Then $H$ is a continuous semimetric of $X$, and so by the assumption $\mu$ has bounded  complexity w.r.t. $\{\overline{H}_{F_n}: n\in \N\}$, hence $h$ is almost periodic in $L^2(X,\mu)$ by Theorem \ref{201904291}.
As $C(X)$ is dense in $L^2(X,\mu)$, one has that each function in $L^2(X,\mu)$ is almost periodic, thus $\mu$ has discrete spectrum.
\end{proof}

The equivalence of $\eqref{mfirst}\Leftrightarrow \eqref{mthird}\Leftrightarrow \eqref{mfourth}$ in Theorem \ref{main} is shown by the following result, which reveals the relationship between discrete spectrum and complexity functions
induced by a sequence of semimetrics
generated by
the Hamming distance of names of partitions (and of partitions of two elements). This generalizes \cite[Proposition 3]{Fer} and \cite[Corollary 3.2]{Yu2018}
 from $\Z$-actions to amenable group actions.

\begin{thm}\label{20190428}
The following statements are equivalent:
\begin{enumerate}
\item \label{0719-1}  $\mu$ has discrete spectrum.

\item \label{0719-2} For any finite partition $\alpha=\{A_1,\cdots,A_l\}$ of $X$,  $\mu$ has bounded  complexity w.r.t. $\{\overline{H^\alpha}_{F_n}: n\in \N\}$, where $H^\alpha(x,y)=\frac{1}{2}\sum_{i=1}^l |1_{A_i}(x)-1_{A_i}(y)|$.

\item \label{0719-3} For any partition $\alpha=\{A, X\setminus A\}$ of $X$,  $\mu$ has bounded  complexity w.r.t. $\{\overline{H^\alpha}_{F_n}: n\in \N\}$, where $H^\alpha(x,y)= |1_{A}(x)-1_{A}(y)|$.
\end{enumerate}
\end{thm}
\begin{proof}
$\eqref{0719-1}\Rightarrow \eqref{0719-2}$
Again by arguments similar to the proof of Theorem \ref{201904291}, we may assume without loss of generality that the F${\o}$lner sequence $\{F_n: n\in \N\}$ is tempered.
 We fix $\ep> 0$. We also fix each $i= 1, \cdots, l$, then set $H^i(x,y)=|1_{A_i}(x)-1_{A_i}(y)|$ for all $x, y\in X$.
 Since $1_{A_i}$ is almost periodic,
 by Theorem \ref{20190320} there exists $\delta_i>0$ and compact $C_i\subset X$
  with $\mu(C_i)>1-\frac{\ep}{l}$ such that
   $\overline{H^i}_{F_n}(x,y)<\frac{\ep}{l}$ for each $n\in \N$ and all $x,y\in C_i$ with $d(x,y)<\delta_i$.
 Let $\delta=\underset{1\le i\le l}{\min} \delta_i> 0$ and $C=\underset{i=1}{\overset{l}{\bigcap}} C_i\subset X$ which is a compact set. Then $\mu(C)>1-\ep$ and, for each $n\in \N$ and all $x,y\in C$ with $d(x,y)<\delta$,
 $$\overline{H^\alpha}_{F_n}(x,y)=\frac{1}{2}\sum_{i=1}^l \overline{H^i}_{F_n}(x,y)<\frac{1}{2}\cdot l\cdot \frac{\ep}{l}= \frac{\ep}{2}.$$
 From this it is easy to obtain that the sequence $\{C(\overline{H^\alpha}_{F_n},\ep): n\in \N\}$ is bounded.

$\eqref{0719-2}\Rightarrow \eqref{0719-3}$ is obvious. It remains to prove $\eqref{0719-3}\Rightarrow \eqref{0719-1}$.

For each $A\in \mathcal{B}_X$, by the assumption that $\mu$ has bounded  complexity w.r.t. $\{\overline{H^\alpha}_{F_n}: n\in \N\}$ for $\alpha= \{A, X\setminus A\}$. Observe that in this case $H^\alpha(x,y)=|1_{A}(x)-1_{A}(y)|$
for all $x,y\in X$, and so by Theorem \ref{201904291} the function $1_A$ is almost periodic. It is trivial to see that the linear sum of finitely many almost periodic functions in $L^2 (X, \mu)$ is still almost periodic. Since the linear span of $\{1_A: A\in \mathcal{B}_X\}$ is dense in $L^2 (X, \mu)$, we have that each function in $L^2 (X, \mu)$ is almost periodic, and so $\mu$ has discrete spectrum.
\end{proof}

The equivalence of $\eqref{mfirst}\Leftrightarrow (5)\Leftrightarrow (6)$ in Theorem \ref{main}, under the condition that the F${\o}$lner sequence $\{F_n: n\in \mathbb{N}\}$ is tempered, follows from the following result, which extends \cite[Theorem 3.1]{Yu2018} and \cite[Theorem 4.2]{HLTXY}.

\begin{thm}
Assume that the F${\o}$lner sequence $\{F_n: n\in \mathbb{N}\}$ is tempered. Then the following statements are equivalent:
\begin{enumerate}

\item \label{07191} $\mu$ has discrete spectrum.

\item \label{07192} For any finite partition $\alpha=\{A_1,\cdots,A_l\}$ of $X$ and each $\varepsilon>0$, there exist finitely many Borel subsets $B_1,\cdots,B_k$
such that $\mu(\bigcup_{i=1}^k B_i)>1-\varepsilon$ and $\overline{H^\alpha}_{F_n}(x,y)<\varepsilon$ for all $x,y\in B_i$ (for some $i= 1, \cdots, k$) and any $n\in \N$.

\item \label{07193} $\mu$ is equicontinuous in the mean w.r.t. $\{F_n: n\in \N\}$.

\item \label{07194} $\mu$ is mean equicontinuous w.r.t. $\{F_n: n\in \N\}$.
\end{enumerate}
\end{thm}
\begin{proof}
The equivalence of $\eqref{07191}\Leftrightarrow \eqref{07192}$ follows directly from
Lemma \ref{20190324} and Theorem \ref{20190428}.
The direction $\eqref{07193}\Rightarrow \eqref{07194}$ is trivial, and so it remains to prove $\eqref{07192}\Rightarrow \eqref{07193}$ and $\eqref{07194}\Rightarrow \eqref{07191}$.

$\eqref{07192}\Rightarrow \eqref{07193}$
Fix $\tau> 0$, and assume without loss of generality $\text{diam} (X)\leq 1$, where $\text{diam} (\bullet)$ denotes the diameter of a subset $\bullet$.
For each $m\in \N$ we may take $\alpha^m=\{A_m^1, \cdots,A_m^{t_m}\}\in \mathcal{P}_X$ such that $\text{diam} (A_m^i)<\frac{\tau}{2^m}$ for all $i= 1, \cdots, t_m$. By the assumption,
there exist finitely many Borel subsets $B_{m}^1, \cdots, B_{m}^{n_{m}}$
such that $\mu(\bigcup_{i=1}^{n_m} B_m^i)>1-\frac{\tau}{2^m}$ and $\overline{H^{\alpha^m}}_{F_n}(x,y)<\frac{\tau}{2^m}$ for all $x,y\in B_m^i$ (for some $i= 1, \cdots, n_m$) and any $n\in \N$. Furthermore, by the regularity of $\mu$, we may assume that each $B_{m}^i, i= 1, \cdots, n_m$ is a compact set. Then
 $\mu(K_{m})>1-\frac{\tau}{2^m}$ for each $m\in \N$ and hence $\mu(K)>1-\tau$, where $K=\overset{\infty}{\underset{ m=1}{\bigcap}} K_m$ and
 $K_{m}=  \overset{n_{m}}{\underset{ i=1}{\bigcup}} B_{m}^i$ for each $m\in \N$.

Now we show that $K$ is equicontinuous in the mean w.r.t. $\{F_n: n\in \N\}$. Clearly, $K\subset X$ is compact.
 Assume the contrary that $K$ is not equicontinuous in the mean w.r.t. $\{F_n: n\in \N\}$, and so there exists $\eta>0$ such that for any $k\in \N$ there are $x_k,y_k\in K$
 and $s_k\in \N$ such that $d(x_k,y_k)<\frac{1}{k}$ and $\overline{d}_{F_{s_k}}(x_k,y_k)\geq \eta$.
 As $K$ is compact, we may assume (by taking a subsequence if necessary) $\lim_{k\rightarrow \infty} x_k= x$ for some $x\in K$,
and hence $\lim_{k\rightarrow \infty} y_k= x$.

 Fix each $k\in\N$. By the triangle inequality, either
 $\overline{d}_{F_{s_k}}(x_k,x)\geq \frac{\eta}{2}$ or $\overline{d}_{F_{s_k}}(y_k,x)\geq \frac{\eta}{2}$.
 Again
 we may assume $\overline{d}_{F_{s_k}}(x_k,x)\geq \frac{\eta}{2}$ (by choosing a subsequence if necessary).
 Set $E_m^k=\{g\in F_{s_k}\colon \alpha_g^m(x)\neq \alpha_g^m(x_k)\}$ for each $m\in \N$, then $\overline{H^{\alpha^m}}_{F_{s_k}}(x_k,x)=\frac{1}{|F_{s_k}|}\cdot |E_m^k|$.
Recall $\text{diam} (X)\leq 1$ and $\text{diam} (A_m^i)<\frac{\tau}{2^m}$ for each $i= 1, \cdots, t_m$. Thus if $g\in E_m^k$ then $d(gx,gx_k)\leq 1$ and if $g \in F_{s_k}\setminus E_m^k$ then $d(gx,gx_k)\leq \frac{\tau}{2^m}$. So, for any $m\in \N$,
 $$\frac{1}{|F_{s_k}|}\cdot \left(|E_m^k|\cdot 1+ (|F_{s_k}|-|E_m^k|)\cdot \frac{\tau}{2^m}\right)\geq
 \overline{d}_{F_{s_k}}(x_k,x)\geq \frac{\eta}{2}.$$
 Take $M\in \N$ with $\frac{\tau}{2^{M}}< \frac{\eta}{4}$. Then $|E_{M}^k|> \frac{\eta}{4}\cdot |F_{s_k}|$, and hence $\overline{H^{\alpha^{M}}}_{F_{s_k}}(x_k,x)> \frac{\eta}{4}$.
 Since $\{x_k: k\in \N\}\subset K$ and $K\subset K_{M}$, we may assume $\{x_k: k\in \N\}\subset B_{M}^i$ for some $i= 1, \cdots, n_{M}$ (choosing subsequence if necessary).
 Since $B_{M}^i$ is compact, one has $x\in B_{M}^i$,
and then $ \overline{H^{\alpha^{M}}}_{F_{s_k}}(x_k,x)< \frac{\tau}{2^{M}}<\frac{\eta}{4} $, a contradiction.

$\eqref{07194}\Rightarrow \eqref{07191}$ By Theorem \ref{20190427} it suffices to prove that once $\mu$ is mean equicontinuous w.r.t. $\{F_n: n\in \N\}$ then $\mu$ has bounded complexity w.r.t. $\{\overline{d}_{F_n}: n\in \N\}$.


Fix any $\ep>0$. Then there is compact $K\subset X$
such that $\mu(K)>1-\frac{\ep}{2}$ and $K$ is mean equicontinuous w.r.t. $\{F_n: n\in \N\}$, and so there exists $\delta>0$ such that if $x,y\in K$ satisfy $d(x,y)<\delta$ then $\limsup_{n\ra \infty} \overline{d}_{F_n} (x, y)<\ep/4$.
By the compactness of $K$,
there exists $\{x_1, \cdots, x_m\}\subset K$ such that
$K\subset \bigcup_{i= 1}^m B (x_i,\delta)$.
For each $j=1,\cdots,m$ and any $N\in\N$, let
\[A_N(x_j)=\left\{y\in B(x_j,\delta)\cap K\colon
\overline{d}_{F_n} (x_j, y)<\ep/2\
\text{for all}\ n\ge N\right\}.
\]
It is easy to see that, for each $j=1,\cdots,m$,
$\{A_N(x_j): N\in \N\}$ is an increasing sequence of Borel subsets and $B(x_j,\delta)\cap K=\bigcup_{N=1}^\infty A_N(x_j)$.
Thus we may choose $N_1\in\N$ and compact $K_1\subset A_{N_1}(x_1)$
such that
\[
\mu(K_1)>\mu(B(x_1,\delta)\cap K)-\frac{\ep}{2m}.
\]
And then choose $N_2\in\N$ and compact $K_2\subset A_{N_2}(x_2)$
such that
\[
K_1\cap K_2=\emptyset
\text{ \ and \ }
\mu(K_1\cup K_2)>\mu((B(x_1,\delta)\cup B(x_2,\delta))\cap K)-\frac{2\ep}{2m}.
\]
By induction,
we can choose $N_j\in \N$ and compact $K_j\subset A_{N_j}(x_j)$ for all $j=1,\cdots,m$
such that $K_1, \cdots, K_m$ are pairwise disjoint and $\mu (K_0)> \mu(K)-\frac{\ep}{2}>1-\ep$, where $K_0=\bigcup_{j=1}^m K_j$.

Now set $N=\max\{N_1, \cdots, N_m\}$.
Clearly there exists $\delta_2>0$ such that if $x,y\in X$ satisfy $d(x,y)<\delta_2$ then $\max_{g\in F_{n}}d(gx,gy)<\ep$ for all $n= 1, \cdots, N$.
Let
\[\delta_1=\min \left\{\underset{ 1\leq i\not=j \leq m}{\min}d(K_i,K_j),\ \ \delta_2\right\}> 0.\]
By the compactness of $K_0$, there exists finite $H\subset K_0$ such that $K_0\subset \bigcup_{x\in H}B(x,\delta_1)$.
Fix $n\geq 1$ and $y\in K_0$, and we choose $x\in H$ with $d(x,y)<\delta_1$.
If $n<N$, then $\overline{d}_{F_n}(x,y)\leq \max_{g\in F_{n}}d(gx,gy)<\ep$.
If $n\geq N$, by the construction of $\delta_1$ there exists $j\in\{1,\cdots,m\}$
such that both $x$ and $y$ belong to $K_j\subset A_{N_j}(x_j)$, and hence (observing $n\geq N\ge N_j$)
\[
\overline{d}_{F_n} (x, y)\le
\overline{d}_{F_n} (x, x_j)+ \overline{d}_{F_n} (x_j, y)
<\ep/2+\ep/2=\ep.
\]
Summing up, for each $n\in \N$, we have $\overline{d}_{F_n}(x,y)<\ep$ and then
\[K_0\subset K_\ep\ \text{where}\ K_\ep= \bigcup_{x\in H} B_{\overline{d}_{F_n}}(x,\ep),\]
in particular, $\mu (K_\ep)
\geq \mu(K_0)>1-\ep$. Thus $C (\overline{d}_{F_n}, 2 \ep)\le |H|$ for all $n\in \N$. By arbitrariness of $\ep> 0$ we have that $\mu$ has bounded complexity w.r.t. $\{\overline{d}_{F_n}: n\in \N\}$.
\end{proof}

\bibliographystyle{amsplain}




\end{document}